\documentclass{amsart}

\usepackage{amsmath,amssymb}
\usepackage{enumitem}
\usepackage{xcolor}

\newtheorem{theorem}{Theorem}[section]

\newtheorem{lemma}[theorem]{Lemma}
\newtheorem{proposition}[theorem]{Proposition}

\newtheorem{theoremA}{Theorem}

\theoremstyle{definition}
\newtheorem{definition}[theorem]{Definition}

\newtheorem{reduction}{Reduction}

\theoremstyle{remark}

\numberwithin{figure}{section}

\newcommand{\BB}[1]{{\mathbb #1}}

\newcommand{\OP}{\operatorname}

\def\Fred{F_{\text{red}}}
\def\Inj{\operatorname{Inj}}
\newcommand\IW[1]{\langle#1\rangle}
\def\rank{\operatorname{rank}}
\def\sk{\operatorname{sk}}
\def\ss{\mathfrak{ss}}

\def\dell{\overset{\leftarrow}{\delta}}
\def\delr{\overset{\rightarrow}{\delta}}

\def\FF{\mathbb{F}}
\def\ZZ{\mathbb{Z}}

\title[Non-derangements]{The complex of injective words of permutations which are not derangements is contractible}

\author{Assaf Libman}
\address{Institute of Mathematics, University of Aberdeen, Aberdeen, UK}
\email{a.libman@abdn.ac.uk}



\begin{document}

\begin{abstract}
Let $D_n \subseteq \Sigma_n$ be the set of derangements in the symmetric group. 
We prove that the complex of injective words generated by $\Sigma_n \!\!\!\setminus\!\! D_n$  is contractible.
This gives a conceptual explanation to the well known fact that the complex of injective words generated by $\Sigma_n$ is homotopy equivalent to the wedge sum $\underset{|D_n|}{\bigvee} S^{n-1}$.
\end{abstract}

\maketitle

\section{Introduction and the main result}

An injective word in the alphabet $[n]=\{1,\dots,n\}$ is a sequence $w=\IW{i_1 \dots i_k}$ of {\em distinct} elements $i_1,\dots,i_k \in [n]$.
A subword is a subsequence $v$ of $w$. 
Elements of the symmetric group $\Sigma_n$ correspond to injective words of length $n$.

Given $S \subseteq \Sigma_n$ let $X(S)$ denote the set of all the non-empty subwords of the elements in $S$, partially ordered by the subword relation $\preceq$.
Let $|X(S)|$ denote the cell complex with one cell $\Delta(w) \cong \Delta^{k-1}$, the standard $(k-1)$-simplex, for each word $w=\IW{i_1\dots i_k}$ of length $k$ in $X(S)$ and where $\Delta(v)$ is identified with the obvious face of $\Delta(w)$ if $v \preceq w$.
See Section \ref{sec:preliminaries} below for details. 
The study of the cell complex associated to the poset of all injective words dates back to Farmer \cite{MR0529895} and Bj\o rner \cite{MR1035492}, and in the generality of this paper it appears, e.g, in \cite{MR2552261} and \cite{MR4071308}.

When $S=\Sigma_n$ the poset $X(\Sigma_n)$ consists of all non-empty injective words (in the alphabet $[n]$).
It arises  in various branches of mathematics such as homological algebra e.g \cite{MR2043333}, and homological stability e.g \cite{MR2155519}, \cite{MR0586429} and others.
It was shown by Farmer \cite[Theorem 5]{MR0529895} that $| X(\Sigma_n)|$ has the homology of a wedge sum of spheres $\vee_k S^{n-1}$.
Farmer's result was improved by Bj\"orner and Wachs in \cite[Theorem 6.1]{MR0690055} who showed that there is, in fact, a homotopy equivalence $|X(\Sigma_n)| \simeq \vee_k S^{n-1}$.
Various other proofs of different flavours have been found for this result.
A topological proof was given by Randal-Williams \cite[Proposition 3.2]{MR3032101}.

It follows from a straightforward calculation of the Euler characteristic  that $|X(\Sigma_n)| \simeq \bigvee_{|D_n|} S^{n-1}$ where $D_n \subseteq \Sigma_n$ is the set of derangements, i.e permutation with no fixed points.
To date, however, there was no known natural one-to-one correspondence between the set of derangements and the factors in this wedge sum decomposition.
The purpose of this note is to prove the following result.

\begin{theoremA}\label{theorem:contractible}
Set $P_n=\Sigma_n \!\!\setminus \!\! D_n$.
Then $|X(P_n)|$ is contractible.
\end{theoremA}

By Proposition \ref{prop:skn-2}  $\sk_{n-2} |X(P_n)| = \sk_{n-2} |X(\Sigma_n)|$, hence
\[
| X(\Sigma_n)| \simeq | X(\Sigma_n)| / | X(P_n)| 
=\bigvee_{|D_n|} S^{n-1}.
\]
Several remarks are in place.
Our proof of the contractibility of $|X(\Sigma_n\!\!\setminus \!\! D_n|$ relies on Bj\"orner-Wachs result that $|X(\Sigma_n)|$ is $(n-2)$-connected.
Thus, we cannot boast giving an independent proof of their result.

Unlike Bj\o rner and Wachs, our proof of Theorem \ref{theorem:contractible}  does not provide shelling for the poset $\Sigma_n\!\!\setminus\! D_n$. 
In fact, one can check that the recursive coatom ordering they found for $X(\Sigma_n)$ does not restrict to $X(\Sigma_n\!\!\setminus \!\!D_n)$.
We do not know if  $X(\Sigma_n \!\!\setminus\!\! D_n)$ is shellable.

\section{Preliminaries}\label{sec:preliminaries}

\subsection*{Semi simplicial sets}
A semi-simplicial set $X$ is a collection of sets $(X_n)_{n=0}^\infty$ and maps $\partial_i \colon X_k \to X_{k-1}$ where $0 \leq i \leq k$, called face maps, satisfying the identities $\partial_j \circ \partial_i = \partial_{i} \circ \partial_{j+1}$ if $i \leq j$.
The geometric realisation of $X$ denoted $\| X\|$ is the quotient space of $\coprod_k X_k \times \Delta^k$, where $\Delta^k$ is the standard $k$-simplex, under the obvious identification of $\{\partial_ix\} \times \Delta^{k-1}$ with the $i^\text{th}$ face of $\{x\} \times \Delta^k$ for any $x \in X_k$.
Thus, $\| X\|$ is a cell complex with one $k$-cell for each $x \in X_k$.
See e.g \cite{MR0084138}, \cite{MR0300281}.

The homology of $\|X\|$ can be computed by means of the following chain complex $C_*(X)$.
The group of $k$-chains is 
\[
C_k(X)=\ZZ[X_k], 
\]
the free abelian group with basis $X_k$.
The face maps $\partial_i \colon X_k \to X_{k-1}$ extend to unique homomorphisms $\partial_i \colon C_k(X) \to C_{k-1}(X)$ and the differential in the chain complex $d_k \colon C_k(X) \to C_{k-1}(X)$ is given by $d_k = \sum_{j=0}^k (-1)^j \partial_j$.
See e.g \cite[Chapter 2.1]{MR1867354}.

\subsection*{Complexes of injective words}
As in the introduction, we fix the alphabet 
\[
[n]=\{1,\dots,n\}
\]
and let $\Inj(n)$ denote the set of all {\em non-empty} injective words $w=\IW{i_1\dots i_k}$ in the alphabet $[n]$ (i.e $i_1, \dots, i_k$ are distinct), partially ordered by the relation of subwords $v \preceq w$.
For $1 \leq k \leq n$ denote
\[
\Inj_k(n) = \{ s \in \Inj(n) : \text{$s$ has length $k$}\}
\]
There are ``deletion maps'' $\partial_i \colon \Inj_k(n) \to \Inj_{k-1}(n)$, where $1 \leq i \leq k$,
\[
\partial_i \IW{s_1\cdots s_k} = \IW{s_1 \cdots \widehat{s_i} \cdots s_k}.
\]
Suppose that $T \subseteq \Inj(n)$ is a subposet which is closed under the formation of subwords.
Set $T_k = T \cap \Inj_k(n)$, the subset of words of length $k$ in $T$.
Then the deletion maps $\partial_i$ restrict to maps $\partial_i \colon T_k \to T_{k-1}$.
We obtain a semi-simplicial set $\ss(T)$ where 
\[
\ss(T)_k = T_{k+1}
\]
and the face maps 
\[
\ss(T)_k \xrightarrow{\partial_i} \ss(T)_{k-1} \qquad (0 \leq i \leq k)
\] 
are the deletion maps $T_{k+1} \xrightarrow{\partial_{i+1}} T_k$. 
We denote
\[
|T| = \| \ss(T)\|
\]
and  call it  the {\em complex of injective words associated to $T$}.
It is a cell complex formed of one cell of dimension $k$ for each injective word $t \in T$ of length $k+1$.

It is worthwhile to remark that the cell complex $|T|$ is homeomorphic (albeit not in a natural way) to the nerve of the poset $T$ (considered as a small category).
Thus, our notation does not conflict with the standard one $|T|$ for the nerve.

By the subsection above, the cellular chain complex $C_*(T)$ that computes the homology of $|T|$ has the form 
\begin{eqnarray}\label{eq:chain complex C(T)}
&& C_k(T)=\ZZ[T_{k+1}] 
\\
\nonumber
&& C_k(T) \xrightarrow{d_k} C_{k-1}(T) \qquad \text{  where  } \qquad d_k = \sum_{j=1}^{k+1} (-1)^{j-1} \partial_j.
\end{eqnarray}

\subsection*{Permutation and injective words}
There is a bijection 
\[
\Sigma_n \approx \Inj_n(n), \qquad \sigma \mapsto \IW{\sigma(1) \,\cdots \, \sigma(n)}.
\]
%
For $S \subseteq \Sigma_n$ let $X(S)$ be the smallest subset of $\Inj(n)$ which contains $S$ and is closed under formation of subwords. 
That is, 
\[
X(S) = \{ t \in \Inj(n) \ : \ \text{$t \preceq s$ for some $s \in S$}\}.
\]
Let $X_k(S)$ be the subset consisting of the words of length $k$.
We call $|X(S)|$ the {\em complex of injective words generated by $S$}.
Its homology groups are computed by the chain complex $C_*(X(S))$ described in \eqref{eq:chain complex C(T)}, with $C_k(X(S))=\ZZ[X_{k+1}(S)]$.

\section{Theorem \ref{theorem:contractible} for $n=2,3$}\label{section:n=2,3}

In this section we prove Theorem \ref{theorem:contractible} for $n=2,3$.
The case $n=3$ is an excellent example that illustrates a phenomenon that the cell complex $X(\Sigma_n \!\! \setminus \!\!D_n)$ exhibits that is the key to the proof of the theorem in the general case.

\subsection*{The case $n=2$}
It is clear that $P:=\Sigma_2\! \! \setminus\! D_2 = \{ \OP{id}\}$.
As a collection of injective words, $P = \{\IW{12}\}$.
By inspection $X(P)=\{ \IW{12}, \IW{1},\IW{2}\}$ so  $|X(P)| \cong \Delta^1$ which is contractible.

\subsection*{The case $n=3$}
Set $P=\Sigma_3 \!\!\setminus \!D_3$.
Thus, $P$ is the set of permutations in $\Sigma_3$ that fix at least one letter.
In the next table we list the elements of $P$, presented as injective words, in lexicographical order.
In the second column we list all the elements of $X_2(P)$ and specify which elements in $P$ it is incident with.

\begin{tabular}{|l|ll|ll|}
\hline
$X_3(P)$              & $X_2(P)$   &  incident with                             & $X_1(P)$         & incident with
\\
\hline
$\sigma_1=\IW{123}$      &  $\IW{12}$    & $\sigma_1,\sigma_2$                  & $\IW{1}$         & $\IW{12}, \IW{13}, \IW{21}, \IW{31}$ 
\\
$\sigma_2=\IW{132}$      &  $\IW{13}$    & $\sigma_1,\sigma_2,\sigma_3$         & $\IW{2}$         & $\IW{12}, \IW{21}, \IW{23}, \IW{32}$
\\
$\sigma_3=\IW{213}$      &  $\IW{21}$    & $\sigma_3,\sigma_4$                 & $\IW{3}$          & $\IW{13}, \IW{23}, \IW{31}, \IW{32}$
\\
$\sigma_4=\IW{321}$      &  $\IW{23}$    & $\sigma_1,\sigma_3$                 &                   &
\\
                      &  $\IW{31}$    & $\sigma_4$                            &                    &
\\
                      &  $\IW{32}$    & $\sigma_2,\sigma_4$                    &                    &
\\
\hline
\end{tabular}

We see that the $1$-cell indexed by $\IW{31}$ is an ``exposed'' face of the $2$-cell indexed by $\sigma_4$ in the sense that no other $2$-cell in $X(P)$ contains it.
We can therefore collapse the interiors of the $1$-simplex $\IW{31}$ and the $2$-simplex $\IW{321}$ onto the boundary of the latter.
Hence $|X(P)|$ is homotopy equivalent (via simple homotopy) to the cell complex associated with the subposet $Y =X(P) \setminus \{ \IW{321},\IW{31}\}$ described in the table below.

\begin{tabular}{|l|ll|ll|}
\hline
$Y_3$              & $Y_2$   &  incident with              & $Y_1$ & incident with 
\\
\hline
$\sigma_1=\IW{123}$      &  $\IW{12}$    & $\sigma_1,\sigma_2$                & $\IW{1}$ & $\IW{12}, \IW{13}, \IW{21}$ 
\\
$\sigma_2=\IW{132}$      &  $\IW{13}$    & $\sigma_1,\sigma_2,\sigma_3$       & $\IW{2}$ & $\IW{12}, \IW{21}, \IW{23}, \IW{321}$
\\
$\sigma_3=\IW{213}$      &  $\IW{21}$    & $\sigma_3$                        & $\IW{3}$ & $\IW{13}, \IW{23}, \IW{32}$
\\
                      &  $\IW{23}$    & $\sigma_1,\sigma_3$               &       &
\\
                      &  $\IW{32}$    & $\sigma_2$                        &       &
\\
\hline
\end{tabular}

This time the $1$-cell indexed by $\IW{21}$ is an exposed face of the $2$-cell $\sigma_3$ and the $1$-cell $\IW{32}$ is an exposed face of the $2$-cell $\sigma_2$.
By collapsing these $2$-cells onto their boundaries we are left with a new cell complex $|Z|$ which is homotopy equivalent to $|Y|$ and is the cell complex associated to the poset of injective words $Z = Y \setminus \{\sigma_2,\sigma_3, \IW{21},\IW{32}\}$ described in the next table.

\begin{tabular}{|l|ll|ll|}
\hline
$Z_3$              & $Z_2$   &  incident with                          & $Z_1$    & incident with 
\\
\hline
$\sigma_1=\IW{123}$      &  $\IW{12}$    & $\sigma_1$                         & $\IW{1}$  & $\IW{12}, \IW{13}$ 
\\
                      &  $\IW{13}$    & $\sigma_1$                        & $\IW{2}$  & $\IW{12}, \IW{23}$
\\
                      &  $\IW{23}$    & $\sigma_1$                        & $\IW{3}$  & $\IW{13}, \IW{23}$ 
\\
\hline
\end{tabular}

\noindent
Clearly $|Z|$ is homeomorphic with $\Delta^2$, proving that $|X(\Sigma_3 \!\!\setminus\!\! D_3)|$ is contractible.

\subsection*{The case $n > 3$}
A similar (and very illuminating!) exercise, albeit lengthy, shows that $|X(\Sigma_4\!\! \setminus \!\!D_4)|$ is contractible.
One simply removes repeatedly top dimensional cells together with their exposed faces until one is left with a point.
The surprise is that this process never stops, namely at each iteration the resulting cell complex contains an exposed face of a top dimensional cell.

While this goes well for $n=3,4$, for $n >4$ we are unable to prove that the iterative process above of removing top cells always yields a cell complex with top dimensional cells containing exposed faces.
In practice, what we will prove is that $H_{n-1}(|X(\Sigma_n \!\!\setminus \!\!D_n)|;R)=0$  for any coefficient ring $R$ by repeatedly showing that more and more top $(n-1)$-cells cannot be basis elements used in any $(n-1)$-cycle, thus showing that $Z_{n-1}(X(\Sigma_n \!\!\setminus \!\!D_n);R)=0$.
As we will see this homological calculation is enough to prove Theorem \ref{theorem:contractible}.

\section{Proof of Theorem \ref{theorem:contractible}}

In Section \ref{section:n=2,3} we have proved the theorem for $n=2,3$.
Throughout this section we therefore fix some $n \geq 4$ and prove Theorem \ref{theorem:contractible}.
Set
\[
S = \{ \text{all injective words of length $n$ in the alphabet $[n]$} \} \approx \Sigma_n.
\]
We will freely interchange the roles of $\Sigma_n$ and $S$ as we deem convenient.
Let $D_n \subseteq \Sigma_n$ denote the set of derangement, i.e permutations with no fixed points.
Set
\[
P=\Sigma_n \setminus D_n.
\]
Thus, $P$ consists of the permutations that fix at least one letter.
Written as injective words,
\[
P = \{\IW{s_1\cdots s_n} \in S : \text{$s_k=k$ for some $1 \leq k \leq n$}\}.
\]
Our goal is to prove that $|X(P)|$ is contractible.

Let $\sk_{d}X$ denote the $d$-skeleton of a cell complex $X$.

\begin{proposition}\label{prop:skn-2}
$\sk_{n-2} |X(P)| = \sk_{n-2} |X(S)|$.
\end{proposition}

\begin{proof}
Any injective word $v \in \Inj(n)$ of length $n-1$ is missing exactly one letter $k \in [n]$.
By inserting $k$ at the $k^\text{th}$ position we form the injective word $w=\IW{v_1\,\cdots \,v_{k-1} \, k \, v_k \, \cdots \, v_{n-1}}$ which clearly belongs to $P$ since $w_k=k$.
It follows that $X_{n-1}(P)=X_{n-1}(S)=\Inj_{n-1}(n)$.
Hence, $X_k(P)=X_k(S)=\Inj_k(n)$ for all $1 \leq k \leq n-1$ and it follows that $\sk_{n-2} |X(P)| = \sk_{n-2} |X(S)|$.
\end{proof}

We will come back to the argument in this proof later. 
Clearly $|X(S)|$ is a cell complex of dimension $n-1$ and, we recall that $|X(S)| \simeq \vee_{|D_n|} S^{n-1}$.
In particular $|X(S)|$ is $(n-2)$-connected.
One easy consequence of this and Proposition \ref{prop:skn-2} is that
\begin{multline*}
\chi(|X(P)|) =
\chi(|X(S)|) - (-1)^{n-1}(|S|-|P|) \\
=
(1+(-1)^{n-1}|D_n|) - (-1)^{n-1}|D_n| = 1.
\end{multline*}
Since $n \geq 4$, the cellular approximation theorem and Proposition \ref{prop:skn-2} imply that $\sk_{n-2}|X(P)|$ is simply connected, and hence $|X(P)|$ is simply connected. 
It also follows from Proposition \ref{prop:skn-2} that
\begin{equation}\label{eq:H_*(P)=H_*(S)=0 *<=n-3}
\tilde{H}_k(|X(P)|;R) = \tilde{H}_k(|X(S)|;R)=0 
\end{equation}
for every $1 \leq k \leq n-3$ and every coefficients ring $R$. 
Therefore, for $R=\ZZ$ or $R=\BB F_p$
\[
\chi(|X(P)|) = 1 + (-1)^{n-2}\rank_R H_{n-2}(|X(P)|;R) + (-1)^{n-1}\rank_R H_{n-1}(|X(P)|;R).
\]
But we have seen that $\chi(|X(P)|)=1$ so it follows that for $R=\ZZ$ or $R=\BB F_p$
\begin{equation}\label{eq:Hn-2 Hn-1}
\rank_R H_{n-2}(|X(P)|;R)  = \rank_R H_{n-1}(|X(P)|;R).
\end{equation}

\begin{reduction}\label{reduction 1}
Theorem \ref {theorem:contractible} follows if for $R=\ZZ$ and $R=\FF_p$
\begin{equation}\label{eq:Hn-1=0}
H_{n-1}(|X(P)|;R) = 0.
\end{equation}
\end{reduction}

\begin{proof}[Proof of Reduction \ref{reduction 1}]
\eqref{eq:H_*(P)=H_*(S)=0 *<=n-3} and \eqref{eq:Hn-2 Hn-1} combined with \eqref{eq:Hn-1=0} imply that $\tilde{H}_*(|X(P)|)=0$, and since $|X(P)|$ is simply connected Whitehead's theorem [REFERENCE] implies that $|X(P)|$ is contractible, as needed.
\end{proof}

We use the chain complex $C_*(X(P);R) = C_*(X(P)) \otimes R$ described in Section \ref{sec:preliminaries} to compute $H_*(|X(P)|;R)$.
For dimensional reasons Reduction \ref{reduction 1} is equivalent to

\begin{reduction}\label{reduction 2}
Theorem \ref {theorem:contractible} follows if for $R=\ZZ$ and $R=\FF_p$
\[
Z_{n-1}(X(P); R)=0.
\]
\end{reduction}

In what follows we will fix the ring $R$ once and for all and write $C_*(X(P))$ for the chain complex $C_*(X(P)) \otimes R$.

We will use the usual notation for intervals $[a,b], [a,b)$;
All intervals in this note are taken in $\mathbb{Z}$.
For example,
\[
[a,b]=\{k \in \ZZ : a \leq k \leq b\}, \qquad [a,b) =\{k \in \ZZ : a \leq k < b\} \qquad \text{etc.}
\]

For the remainder of the proof let $F$ denote the set of all injective words of length $n-1$ in the alphabet $[n]$, i.e
\[
F = \Inj_{n-1}(n). 
\]
Any $t \in F$ has a unique symbol $k \in [n]$ which does not appear in $t$.
We denote for any $1 \leq k \leq n$
\[
F(k) = \{ t \in F : \text{the symbol $k$ is missing from $t$}\}.
\]
Clearly $F$ is the disjoint union $\bigcup_{k=1}^n F(k)$.


\begin{definition}\label{def:sigma_i}
Let $t \in F(k)$.
For any $1 \leq i \leq n$ we define 
\[
\sigma_i(t) = \IW{t_1,\dots,t_{i-1},k,t_i,\dots,t_{n-1}} \in S.
\]
\end{definition}

Thus, $s=\sigma_i(t)$ is the (unique) element in $S$ such that $t \preceq s$ and $s_i=k$:
\begin{equation}\label{eq:sigma_i description}
s_j = \left\{
\begin{array}{ll}
t_j   & \text{if $1 \leq j \leq i-1$} \\
k     & \text{if $j=i$} \\
t_{j-1} & \text{if $i <j \leq n$}.
\end{array}\right.
\end{equation}
It follows that 
\begin{equation}\label{eqn:sigma_k in P}
t \in F(k) \implies \sigma_k(t) \in P.
\end{equation}
In particular, as we have seen in Proposition \ref{prop:skn-2},
\begin{equation}\label{eqn:F=X_n-1(P)}
F = X_{n-1}(P).
\end{equation}

\begin{definition}
Let $t \in F$.
Set
\[
M(t) = \{ 1 \leq i \leq n \ : \ \sigma_i(t) \in P\}.
\]
\end{definition}

It is clear that for any $t \in F$ and any $s \in S$
\[
t=\partial_i s \iff s=\sigma_i(t).
\]
%
Hence, for any $t \in F$,
\begin{equation}\label{eqn:sigma_i t preceq s}
\{ \sigma_i(t) : i \in M(t)\} \ = \ \{ s \in P : t \preceq s\}.
\end{equation}

Any chain $\alpha \in C_{n-1}(X(P))$ is an $R$-linear combination
\begin{equation}\label{eqn:alpha cycle}
\alpha = \sum_{s \in P} \alpha_s \cdot s, \qquad (\alpha_s \in R).
\end{equation}
By the description of the chain complex $C_*(X(P))$, and by \eqref{eqn:F=X_n-1(P)}, the differential  $d(\alpha) \in C_{n-2}(X(P))=\ZZ[F] \otimes R$ has the form 
\begin{eqnarray}\label{eqn:diff alpha}
&& d(\alpha) = \sum_{t \in F} \beta_t \cdot t, \qquad\text{ where } \\ 
\nonumber
&& \beta_t = \sum_{i \in M(t)} (-1)^{i-1} \alpha_{\sigma_i(t)}.
\end{eqnarray}

\begin{definition}\label{def:homologically redundant}
Call $s \in P$ {\em homologically redundant} if with the notation of \eqref{eqn:alpha cycle} $\alpha_s=0$ in any $\alpha \in Z_{n-1}(X(P);R)$.
Set
\[
\Fred = \{  t \in F \ : \ \forall s \in P ( t \preceq s \implies \text{$s$ is homologically redundant})\} 
\]
\end{definition}

\begin{reduction}\label{reduction 3}
Theorem \ref {theorem:contractible} follows if 
\[
\Fred = F.
\]
\end{reduction}

\begin{proof}
Choose some $\alpha \in Z_{n-1}(X(P);R)$ expressed as in \eqref{eqn:alpha cycle}.
Since any $s \in P$ is incident with some $t \in F=\Fred$, it follows that $\alpha_s=0$.
We deduce that $Z_{n-1}(X(P);R)=0$ and then apply Reduction \ref{reduction 2}.
\end{proof}

For the remainder of this section we prove Reduction \ref{reduction 3}.

\begin{lemma}\label{L:G lemma}
Let $t \in F$ and choose some $\ell \in M(t)$.
If for every $i \in M(t) \setminus\{\ell\}$ the simplex $s=\sigma_i(t)$ is incident with a face $t' \in \Fred$ then $t \in \Fred$.
\end{lemma}

\begin{proof}
Consider some $\alpha \in Z_{n-1}(X(P))$ expressed as in \eqref{eqn:alpha cycle} and set $\beta=d(\alpha)$.
Then by \eqref{eqn:diff alpha} and by the hypothesis that every $\sigma_i(t)$ for $i \in M(t)\setminus\{\ell\}$ is homologically redundant it follows that
\[
0 = \beta_t = \sum_{i \in M(t)} (-1)^{i-1} \alpha_{\sigma_i(t)} 
= \sum_{i \in M(t) \setminus \{\ell\}} (-1)^{i-1} \alpha_{\sigma_i(t)}  + (-1)^{\ell-1} \alpha_{\sigma_\ell(t)}
= \pm \alpha_{\sigma_\ell(t)}.
\]
Therefore, $\alpha_{\sigma_i(t)}=0$ for all $i \in M(t)$, so  $t \in \Fred$ by \eqref{eqn:sigma_i t preceq s} and Definition  \ref{def:homologically redundant}.
\end{proof}


\begin{definition}\label{def:J_0 J_1 rho lambda}
Let $t \in F$.
Set
\begin{align*}
 J_0(t) &= \{ 1 \leq j \leq n-1 \ : \ t_j=j\} \\
 J_1(t) &= \{ 1 \leq j \leq n-1 \ : \ t_j=j+1\}
\\
\lambda(t) &= \max \ J_1(t) \cup \{0\}, 
\\
\rho(t) &= \min \ J_0(t) \cup \{n \}.
\end{align*}
\end{definition}

It is clear from the definition that 
\begin{equation}\label{eqn:lambda rho bounds}
0 \leq \lambda(t) \leq n-1 \qquad \text{and} \qquad 1 \leq \rho(t) \leq n.
\end{equation}
It is also clear that $J_0(t) \cup\{n\}$ and $J_1(t) \cup \{0\}$ are disjoint so
\begin{equation}\label{eqn:lambda neq rho}
\lambda(t) \neq \rho(t).
\end{equation}

\begin{lemma}\label{L:description M(t)}
Consider some $t \in F(k)$ and set $\lambda=\lambda(t)$ and $\rho=\rho(t)$.
Then
\[
M(t) = [1,\lambda] \cup [\rho+1,n] \cup \{k\}.
\]
\end{lemma}

\begin{proof}
We start by proving the inclusion $M(t) \supseteq [1,\lambda] \cup [\rho+1,n] \cup \{k\}$.
Write $t=\IW{t_1\,\cdots \,t_{n-1}}$.
First, $k \in M(t)$ since $\sigma_k(t) \in P$ by \eqref{eqn:sigma_k in P}.  

Suppose that $[1,\lambda] \neq \emptyset$ and consider some $i \in [1,\lambda]$.
Then $\lambda \geq 1$ so $J_1(t) \neq \emptyset$ and $\lambda \in J_1(t)$.
Set $s=\sigma_i(t)$.
Notice that $\lambda \leq n-1$ and that by \eqref{eq:sigma_i description}, $s_{\lambda+1}=t_\lambda=\lambda+1$  because $\lambda \in J_1(t)$.
It follows that $s \in P$, hence $ i \in M(t)$.

Suppose that $[\rho+1,n]$ is not empty and consider  some $\rho+1 \leq i \leq n$.
Then $\rho \leq n-1$, so $J_0(t) \neq \emptyset$ and $\rho \in J_1(t)$.
Set $s=\sigma_i(t)$.
Since $\rho \geq 1$ and $\rho \leq i-1$ it follows from \eqref{eq:sigma_i description} that $s_\rho = t_\rho = \rho$ since $\rho \in J_0(t)$.
It follows that $s \in P$, hence $i \in M(t)$.

For the inclusion $M(t) \subseteq [1,\lambda] \cup [\rho+1,n] \cup \{k\}$, suppose that $i \in M(t)$ and set $s=\sigma_i(t)$.
Then $s \in P$, so $s_j=j$ for some $1 \leq j \leq n$.
If $1 \leq j \leq i-1$ then \eqref{eq:sigma_i description} implies that $t_j=s_j=j$ so $j \in J_0(t)$.
Hence, $\rho \leq j  \leq i-1$ and it follows that $i \in [\rho+1,n]$.
If $j=i$ then $s_i=i$ so $i=k$ by \eqref{eq:sigma_i description}.
If  $i+1 \leq j \leq n$ then \eqref{eq:sigma_i description} implies that $t_{j-1}=s_j=j$.
In particular $j-1 \in J_1(t) \neq \emptyset$, hence $\lambda \geq j-1 \geq i$ which implies $i \in [1,\lambda]$.
\end{proof}

\begin{definition}\label{D:excess}
The {\em excess} of $t \in F$ is
\[
\epsilon(t) = \lambda(t)-\rho(t).
\]
Set for every integer $e$ and every $1 \leq k \leq n$
\begin{align*}
& F^e = \{ t \in F : \epsilon(t) \leq e\}, \\
& F^e(k)= F^e \cap F(k).
\end{align*}
\end{definition}

It is clear from \eqref{eqn:lambda rho bounds} that for any $t \in F$,
\[
-n \leq \epsilon(t) \leq n-2,
\]
and therefore that 
\begin{equation}\label{eqn:F^e chain of inclusion}
\emptyset = F^{-n-1} \subseteq F^{-n} \subseteq \dots \subseteq F^{n-2} = F.
\end{equation}

\begin{definition}
Let $(W,\leq_L)$ denote the set of {\em all} words, including the empty word, in the alphabet $[1,n-1]=\{1,\dots,n-1\}$ equipped with the lexicographical (from left to right) order $\leq_L$.
We write elements of $W$ as $r$-tuples $(j_1,\dots,j_r)$.
We also write $w <_L w'$ for $w \leq_L w'$ and $w \neq w'$ (in the lexicographical order on $W$).

Equip $W \times W$ with the usual partial order $\preceq_L$ of the product.
\end{definition}

\begin{definition}
Consider a subset $J \subseteq \{1,\dots,n-1\}$.
Define $\dell(J) \in W$ and $\delr(J) \in W$, words of length $|J|$ in $W$, as follows.
\begin{itemize}
\item
Write $J \cup \{0\} =\{j_1 , \dots , j_p,j_{p+1}\}$ where $p\geq 0$ and $j_1> \dots > j_p>j_{p+1}=0$, and set
\[
\dell(J) = (j_1-j_2\, , \, \cdots, \, j_{p-1}-j_p\, , \, j_p-j_{p+1}).
\]
\item
Write $J \cup \{n\} =\{j_1,\dots,j_p,j_{p+1}\}$ where $p \geq 0$ and $j_1< \dots < j_p<j_{p+1}=n$, and set
\[
\delr(J) = (j_2-j_1, \, \cdots \, , \, j_{p}-j_{p-1} \, , \, j_{p+1}-j_p).
\]
\end{itemize}
\end{definition}

\begin{definition}\label{def:omega(t)}
For any $t \in F$ define $\omega(t) \in W \times W$ by
\[
\omega(t) = (\dell(J_1(t)),\delr(J_0(t))).
\]
\end{definition}


\begin{lemma}\label{L:lambda<rho i<=lambda}
Fix some $1 \leq k \leq n$.
Consider some $t \in F(k)$ and $1 \leq i \leq \lambda(t)$ such that $i \neq k$, and set $s=\sigma_i(t)$.
Then $J_1(t) \cap [i,n-1] \neq \emptyset$ and we may set 
\begin{eqnarray*}
b &=& \min \ J_1(t) \cap [i,n-1] \qquad \text{and}
\\
t' &=& \partial_{b+1}(s).
\end{eqnarray*}
Then the following hold: $i \leq b \leq \lambda(t)$ and $t' \in F(b+1)$ and 
\[
J_0(t') = J_0(t) \bigcap \left([\rho(t),i-1] \cup [\max\{\rho(t),b+1\},n-1]\right).
\]
Moreover, 
\begin{enumerate}[label=(\roman*)]
\item \label{item:b=lambda 1}
if $b=\lambda(t)$ then $\lambda(t')< \lambda(t)$ (possibly $\lambda(t')=0$) and, 

\item \label{item:b<lambda 1} 
if $b<\lambda(t)$ then $\lambda(t')=\lambda(t)$ and $\dell(J_1(t)) <_L \dell(J_1(t'))$ in $W$.
\end{enumerate}
\end{lemma}

\begin{proof}
Set $\lambda=\lambda(t)$ and $\rho=\rho(t)$.
By assumption $\lambda \geq i \geq 1$ so by Definition \ref{def:J_0 J_1 rho lambda} $J_1(t) \neq \emptyset$ and $\lambda \in J_1(t)$, hence $\lambda \in J_1(t) \cap [i,n-1]$.
Thus, $b$ is well defined and 
\[
1 \leq i \leq b \leq \lambda.
\]
By the definition of $b$ and $J_1(t)$, for any $1 \leq j \leq n-1$,
\begin{equation}\label{E:b property 1}
i \leq j <b \implies t_j \neq j+1.
\end{equation}
Additionally, $\lambda \leq n-1$ by \eqref{eqn:lambda rho bounds} so $2 \leq b+1 \leq n$.
Therefore $t'=\partial_{b+1}(s)$ is well defined and by inspection
\[
t'=\IW{t_1,\dots,t_{i-1},k,t_i,\dots,\widehat{t_b},\dots,t_{n-1}}.
\]
Thus,
\begin{equation}\label{E:explicit t' 1}
t'_j=\left\{
\begin{array}{ll}
t_j       & \text{if $1 \leq j \leq i-1$} \\
k         & \text{if $j=i$} \\
t_{j-1}    & \text{if $i+1 \leq j \leq b$} \\
t_j       & \text{if $b< j \leq n-1$}
\end{array}\right.
\end{equation}
By definition, $b \in J_1(t)$, so $t_b=b+1$.
Therefore 
\[
t' \in F(b+1).
\]

We examine $t'_j$ for all $1 \leq j \leq n-1$.
If $b<j \leq n-1$ then by \eqref{E:explicit t' 1} $t'_j=t_j$.
It follows that $J_0(t') \cap (b,n-1] = J_0(t) \cap (b,n-1]$.

If $i+1 \leq j \leq b$ equations \eqref{E:explicit t' 1} and \eqref{E:b property 1} imply $t'_j=t_{j-1} \neq j$. 
If $j=i$ then by \eqref{E:explicit t' 1} and the hypotheses $t'_j=k \neq i =j$.
It follows that $[i,b] \cap J_0(t') = \emptyset$.

If $1 \leq j \leq i-1$ then by \eqref{E:explicit t' 1} $t'_j=t_j$ so $J_0(t') \cap [1,i-1] = J_0(t) \cap [1,i-1]$.
We deduce that
\[
J_0(t') = J_0(t) \bigcap \Big([1,i-1] \cup (b,n-1]\Big).
\]
By Definition \ref{def:J_0 J_1 rho lambda} $J_0(t) \subseteq [\rho,n-1]$, and we deduce that
\[
J_0(t') = J_0(t) \cap \Big( [\rho,i-1] \cup [ \max\{b+1,\rho\},n-1] \Big).
\]
Set $\lambda'=\lambda(t')$.
Consider an arbitrary $\lambda < j \leq n-1$.
Then $b<j \leq n-1$ since $b \leq \lambda$ and $j \notin J_1(t)$ by Definition \ref{def:J_0 J_1 rho lambda}.
By \eqref{E:explicit t' 1} $t'_j=t_j \neq j+1$, hence $j \notin J_1(t')$.
We deduce that
\[
J_1(t') \subseteq [1,\lambda].
\]
We are ready to prove items \ref{item:b=lambda 1} and \ref{item:b<lambda 1}.
Recall that $i \leq b \leq \lambda$.
We examine three cases (Cases 1 and 2 cover $b=\lambda$ and Case 3 covers $b< \lambda$).

\medskip
\noindent
Case 1: $b=\lambda \geq i+1$.
Then $2 \leq b \leq n-1$ and, since $t$ in an injective word and $b \in J_1(t)$ it follows from \eqref{E:explicit t' 1} that $t'_b=t_{b-1} \neq t_b=b+1$.
So $b \notin J_1(t')$.
But $b=\lambda$ so $J_1(t') \subseteq [1,\lambda)$.
In particular $\lambda' = \min \, J_1(t') \cup \{0\} < \lambda$.

\medskip
\noindent
Case 2: $b=\lambda=i$.
Then $t'_i=k$.
In addition, $t_i = t_\lambda = \lambda+1=i+1$ because $\lambda \in J_1(t)$.
Since $t \in F(k)$ it follows that $k \neq t_i=i+1$.
Therefore $t'_i \neq i+1$, namely $i \notin J_1(t')$ and since $\lambda=i$, we get $J_1(t') \subseteq [1,\lambda)$, hence $\lambda' = \min \, J_1(t') \cup \{0\} < \lambda$.

\medskip
\noindent
Case 3: $b<\lambda$.
Write $J_1(t)\cup\{0\} = \{j_1> \cdots > j_p>j_{p+1}=0\}$ where $p \geq 1$ (since $J_1(t) \neq \emptyset$) and $j_1=\max \, J_1(t)=\lambda$ and $b=j_q$ for some $2 \leq q \leq p$ (since $b \in J_1(t)$ and $b<\lambda$).
Write $J_1(t') \cup\{0\} = \{ j'_1> \cdots > j'_r> j'_{r+1}=0\}$.

By \eqref{E:explicit t' 1} $j'_j=j_j$ for any $b<j \leq n-1$.
Hence,
\begin{equation}\label{eqn:J_1 cap [b,n-1] = J_1' cap [b,n-1]}
J_1(t') \cap (b,n-1] = J_1(t) \cap (b,n-1] = \{ j_1> j_2> \dots > j_{q-1}\}.
\end{equation}
In particular 
$r \geq q-1$.

Recall that $i \leq b \leq \lambda$.
Suppose first that $i+1 \leq b \leq \lambda$.
Then $b \geq 2$, and since $t$ is an injective word and $b \in J_1(t)$ it follows from \eqref{E:explicit t' 1} that $t'_b=t_{b-1} \neq t_b=b+1$, and therefore $b \notin J_1(t')$.
Next, suppose that $b=i$.
then by \eqref{E:explicit t' 1} $t'_b=k$.
But $b \in J_1(t)$ so $t_b=b+1$ and also $t_b \neq k$ since $t \in F(k)$.
We deduce that $t'_b \neq b+1$ and therefore $b \notin J_1(t')$.
To conclude, we have just shown that 
\[
b \notin J_1(t').
\]
Thus, $J_1(t') = J_1(t') \cap [1,b) \, \cup \, J_1(t') \cap (b,n-1]$.
Together with \eqref{eqn:J_1 cap [b,n-1] = J_1' cap [b,n-1]}, we get that
\[
J_1(t') \cup \{0\} = \{j_1> \cdots > j_{q-1} > j'_q> \cdots > j'_{r+1}=0\}
\]
where $j'_q=\max J_1(t') \cap [1,b) \cup\{0\}$.
Thus, $j'_q<b=j_q$.
Now,
\begin{eqnarray*}
\dell(J_1(t)) &=& (j_1-j_2\, , \, \cdots \, , \, j_{q-1}-j_q \, , \, \cdots ) \\
\dell(J_1(t')) &=& (j_1-j_2\, , \, \cdots \, , \, j_{q-1}-j'_q \, , \, \cdots )
\end{eqnarray*}
and since $j'_q<j_q$ it follows that $\dell(J_1(t)) <_L \dell(J_1(t'))$.
\end{proof}

The next lemma is similar. 
We give its proof for the sake of completeness.

\begin{lemma}\label{L:lambda<rho i>rho}
Fix some $1 \leq k \leq n$.
Consider some $t \in F(k)$ and some $\rho+1 \leq i \leq n$ such that $i \neq k$.
Set $s=\sigma_i(t)$.
Then $J_0(t) \cap [1,i-1] \neq \emptyset$ and we may set
\begin{eqnarray*}
b &=& \max \, J_0(t) \cap [1,i-1] \\
t' &=& \partial_{b}(s).
\end{eqnarray*}
Then the folowing hold: $\rho(t) \leq b \leq i-1$ and $t' \in F(b)$ and 
\[
J_1(t')=J_1(t) \bigcap \Big( [1,\min\{\lambda(t),b-1\}] \cup [i,\lambda(t)]\Big).
\]
Moreover, 
\begin{enumerate}[label=(\roman*)]
\item 
\label{item:b=rho 2}
if $b=\rho(t)$ then $\rho(t') > \rho(t)$ (possibly $\rho(t')=n$) and, 

\item \label{item:b>rho 2} 
if $b>\rho(t)$ then $\rho(t')=\rho(t)$ and $\delr(J_0(t)) <_L \delr(J_0(t'))$ in $W$.
\end{enumerate}
\end{lemma}

\begin{proof}
Set $\rho=\rho(t)$ and $\lambda=\lambda(t)$.
The assumption $\rho+1 \leq i$ implies $\rho \leq n-1$ so by Definition \ref{def:J_0 J_1 rho lambda} $J_0(t) \neq \emptyset$ and  $\rho \in J_0(t) \cap [1,i-1]$ since $\rho \leq i-1$.
In particular this set is not empty, $b$ is well defined and by \eqref{eqn:lambda rho bounds}
\[
1 \leq \rho \leq b \leq i-1.
\]
By definition of $b$ and of $J_0(t)$, for any $1 \leq j \leq n-1$,
\begin{equation}\label{E:b property 2}
b< j \leq i-1 \implies t_j \neq j.
\end{equation}
Now, $t'=\partial_b(s)$ is well defined since $1 \leq b \leq n-1$, and by inspection
\[
t'=\IW{t_1,\dots,\widehat{t_b},\dots,t_{i-1},k,t_i,\dots,t_{n-1}}.
\]
\begin{equation}\label{E:t' explicit 2}
t'_j =
\left\{
\begin{array}{ll}
t_j        & \text{if $1 \leq j < b$} \\
t_{j+1}     & \text{if $b \leq j \leq i-2$} \\
k          & \text{if $j=i-1$} \\
t_j        & \text{if $i \leq j \leq n-1$}
\end{array}
\right.
\end{equation}
Since $b \in J_0(t)$ we get $t_b=b$ so 
\[
t' \in F(b).
\]
Consider an arbitrary $1 \leq j \leq n-1$.
If $j \geq i$ then by \eqref{E:t' explicit 2} $t'_j=t_j$. 
It follows that  $J_1(t') \cap [i,n-1] = J_1(t) \cap [i,n-1]$.

If $j=i-1$ then $t'_j=k\neq i=j+1$ by \eqref{E:t' explicit 2}.
If $b \leq j \leq i-2$ then $b<j+1 \leq i-1$ so \eqref{E:t' explicit 2} and \eqref{E:b property 2} imply $t'_j=t_{j+1} \neq j+1$.
It follows that $[b,i-1] \cap J_1(t')=\emptyset$.

If $j<b$ then $t'_j=t_j$ so $J_1(t') \cap [1,b)= J_1(t) \cap [1,b)$.
We deduce that
\[
J_1(t') = J_1(t) \bigcap \Big( [1,b) \cup [i,n-1]\Big).
\]
Since  $J_1(t) \subseteq [1,\lambda]$ it follows by Definition \ref{def:J_0 J_1 rho lambda} that
\[
J_1(t') = J_1(t) \bigcap \Big( [1,\min\{\lambda,b-1\}] \cup [i,\lambda]\Big).
\]
Set $\rho'=\rho(t')$.
Consider an arbitrary $1 \leq j<\rho$.
Then $j< b$ (since $\rho \leq b$) and $j \notin J_0(t)$ by Definition \ref{def:J_0 J_1 rho lambda}, so $t'_j=t_j \neq j$ by \eqref{E:t' explicit 2}.
Hence
\begin{equation}\label{eqn:J_0(t') subset [rho,n-1]}
J_0(t') \subseteq [\rho,n-1].
\end{equation}
We are ready to prove items \ref{item:b=rho 2} and \ref{item:b>rho 2}. 
Recall that $\rho \leq b \leq i-1$.
We consider three cases.
Cases 1 and 1 cover $\rho=b$ and Case 3 deals with the case $\rho<b$.

\medskip
\noindent
Case 1: $b=\rho$ and $b \leq i-2$.
Since $t$ is an injective word and $b \in J_0(t)$ we deduce from \eqref{E:t' explicit 2} that $t'_b=t_{b+1} \neq t_b=b$.
Therefore $\rho=b \notin J_0(t')$, hence $J_0(t') \subseteq (\rho,n-1]$ and consequently $\rho' = \min \, J_0(t') \cup \{n \} >\rho$ (possibly $\rho'=n$).

\medskip
\noindent
Case 2: $b=\rho$ and $b = i-1$.
Then $t'_b=t'_{i-1}=k$ by \eqref{E:t' explicit 2}.
But 
$\rho \in J_0(t)$ so $t_\rho=\rho$.
Since $t \in F(k)$ it follows that $k \neq t_{\rho}$.
In particular $t'_{\rho} \neq \rho$ so $\rho \notin J_0(t')$, and it follows that $J_0(t') \subseteq (\rho,n-1]$.
Therefore, $\rho' = \min \, J_0(t') \cup \{n \} > \rho$.

\medskip
\noindent
Case 3: $b>\rho$.
Write $J_0(t) \cup \{n \} =\{j_1<\dots<j_p<j_{p+1}=n\}$ where $p \geq 1$ (since $J_0(t) \neq \emptyset$) and $j_1=\rho$ (since $\rho=\min J_0(t)$).
Also, $b=j_q$ for some $2 \leq q \leq p$ since $b \in J_0(t)$ and $b>\rho$.
Write  $J_0(t') \cup \{n \}=\{j'_1< \dots < j'_r<j'_{r+1}=n\}$.

Consider an arbitrary $1 \leq j<b$.
Then $t'_j=t_j$ by \eqref{E:t' explicit 2} so
\begin{equation}\label{eqn:J_0(t') cap [1,b) = J_0(t) cap [1,b)}
J_0(t') \cap [1,b) = J_0(t) \cap [1,b) = \{ j_1<j_2 < \dots < j_{q-1}\}.
\end{equation}
In particular $r \geq q-1$.
We now examine $t'_b$.
Since $\rho \leq b \leq i-1$ we check two cases.

If $\rho \leq b \leq i-2$ then since $t$ is an injective word and since $b \in J_0(t)$, \eqref{E:t' explicit 2} implies that $t'_b=t_{b+1} \neq t_b=b$ so $b \notin J_0(t')$.
If $b = i-1$ then $t'_b=k$ by \eqref{E:t' explicit 2} and $t_b=b=i-1$ since $b \in J_0(t)$.
Also,$k \neq t_b$ because $t \in F(k)$.
Hence $t'_{i-1} \neq i-1$, i.e $b=i-1 \neq J_0(t')$ in this case too.
We deduce that 
\[
b \notin J_0(t').
\]
Therefore, $J_0(t') \cup\{n\}  = J_0(t') \cap [1,b) \cup J_0(t') \cap (b,1] \cup \{n\}$ and it follows from \eqref{eqn:J_0(t') cap [1,b) = J_0(t) cap [1,b)} that
\[
J_0(t') = \{j_1<\cdots<j_{q-1}<j'_q<\cdots<j'_{r+1}\}
\]
where $j'_q=\min J_0(t') \cap (b,n-1] \cup\{n\}$ (and recall that $q \leq r+1$). 
Thus, $j'_q >b = j_q$.
Since
\begin{eqnarray*}
\delr(J_0(t)) &=& (j_2-j_1\, , \, \cdots \, , \, j_q-j_{q-1}\, , \, \cdots) \\
\delr(J_0(t')) &=& (j_2-j_1\, , \, \cdots \, , \, j'_q-j_{q-1}\, , \, \cdots),
\end{eqnarray*}
and since $j'q>j_q$, it follows that $\delr(J_0(t)) <_L \delr(J_0(t'))$.
\end{proof}

The next two lemmas are the last ingredient in the proof of Reduction \ref{reduction 3}.

\begin{lemma}\label{L:J cap K dell}
Let $\emptyset \neq J \subseteq [1,n-1]$.
Set $\mu = \max J$.
Consider some $1 \leq a < b \leq \mu$ and set $K=[1,a] \cup [b,\mu]$.
Then $\dell(J) \leq_L \dell(J \cap K)$.
\end{lemma}

\begin{proof}
Write $J \cup \{0\} = \{ j_1 > \cdots > j_p > j_{p+1}=0\}$ where $p \geq 1$ (since $J \neq \emptyset$).
Since $j_1=\mu$ and $j_p \geq 1$ it is clear that there are $1 \leq q < r \leq p+1$ such that
\[
J \cap K \cup \{0\} = 
\{ j_1 > \cdots >j_q > j_r > \cdots > j_{p+1}\}.
\]
If $r=q+1$ then $J \cap K = J$, in which case $\dell(J \cap K)=\dell(J)$.
If $r>q+1$ then $j_r<j_{q+1}$, so $j_q-j_{q+1}<j_q-j_r$, and since
\begin{eqnarray*}
\dell(J) &=& (j_1-j_2\, , \, \cdots \, , \, j_{q-1}-j_q \, , \, j_q-j_{q+1} \, , \, \cdots) \\
\dell(J \cap K) &=& (j_1-j_2\, , \, \cdots \, , \, j_{q-1}-j_q \, , \, j_q-j_r \, , \, \cdots) \\
\end{eqnarray*}
it follows that $\dell(J) <_L \dell(J \cap K)$.
\end{proof}

\begin{lemma}\label{L:J cap K delr}
Let $\emptyset \neq J \subseteq [1,n]$.
Set $\mu = \min J$.
Consider some $\mu \leq a < b$ and set $K=[\mu,a] \cup [b,n-1]$.
Then $\delr(J) \leq_L \delr(J \cap K)$.
\end{lemma}

\begin{proof}
Write $J \cup \{n\} = \{ j_1 < \cdots < j_p < j_{p+1}=n\}$ where $p \geq 1$ (since $J \neq \emptyset$).
Since $j_1=\mu$ and $j_p \leq n-1$ it is clear that there are $1 \leq q < r \leq p+1$ such that
\[
J \cap K \cup \{n\} = 
\{ j_1 < \cdots <j_q < j_r < \cdots < j_{p+1}=n\}.
\]
If $r=q+1$ then $J \cap K = J$, in which case $\delr(J \cap K)=\delr(J)$.
If $r>q+1$ then $j_r>j_{q+1}$, so $j_{q+1}-j_{q}<j_r-j_q$, and since
\begin{eqnarray*}
\delr(J) &=& (j_2-j_1\, , \, \cdots \, , \, j_q-j_{q-1} \, , \, j_{q+1}-j_q \, , \, \cdots) \\
\delr(J \cap K) &=& (j_2-j_1\, , \, \cdots \, , \, j_q-j_{q-1} \, , \, j_r - j_q \, , \, \cdots) \\
\end{eqnarray*}
it follows that $\delr(J) <_L \delr(J \cap K)$.
\end{proof}

\begin{proof}[Proof of Reduction \ref{reduction 3}]
Assume false, namely $\Fred \subsetneq F$.
Recall from Definition \ref{D:excess} the filtration $\{F^e\}$ of $F$.
Notice that $F^{-n-1}=\emptyset \subseteq \Fred$ and that by assumption $F^{n-2}=F \nsubseteq \Fred$.
Let $-n \leq e \leq n-2$ be minimal such that $F^e \nsubseteq \Fred$ and set 
\[
\Omega = F^e \setminus \Fred.
\]
It follows from the minimality of $e$ that $F^{e-1} \subseteq \Fred$, hence $\epsilon(t')=e$ for all $t' \in \Omega$.

Among all $t \in \Omega$ choose one for which $\omega(t)$, see Definition \ref{def:omega(t)}, is maximal in the partial order on $W \times W$ (that is, for any $t' \in \Omega$ either $\omega(t)$ and $\omega(t')$ are not comparable or $\omega(t') \preceq_L \omega(t)$).
There is a unique $1 \leq k \leq n$ such that $t \in F^e(k)$.

By \eqref{eqn:sigma_k in P} $k \in M(t)$.
In the remainder of the proof we will show that for every $i \in M(t) \setminus \{k\}$, if we set $s=\sigma_i(t) \in P$ then $t' \preceq s$ for some $t' \in \Fred$.
Lemma \ref{L:G lemma} then shows that $t \in \Fred$ which is a contradiction.
This contradiction completes the proof of Reduction \ref{reduction 3}.

Set $\lambda=\lambda(t)$ and $\rho=\rho(t)$.
By Lemma \ref{L:description M(t)} $M(t) = [1,\lambda] \cup [\rho+1,n] \cup \{k\}$.
Choose some $i \in M(t) \setminus \{k\}$ and consider $s=\sigma_i(t)$.

\medskip
\noindent
{\em Case 1:} $1 \leq i \leq \lambda$ and $i \neq k$.
Set $b$ and $t'$ as in Lemma \ref{L:lambda<rho i<=lambda}.
Thus, $J_1(t) \neq \emptyset$ (since $\lambda >0$) and $i \leq b \leq \lambda$ and $t'=\partial_{b+1}\sigma_i(t)$ and $t' \in F(b+1)$.
Set $\rho'=\rho(t')$ and $\lambda'=\lambda(t')$.
By Lemma \ref{L:lambda<rho i<=lambda} $J_0(t')=J_0(t) \cap K$ where 
\[
K=[\rho,i-1] \cup [\max\{\rho,b+1\},n-1].
\]
By Definition \ref{def:J_0 J_1 rho lambda} $\rho' \geq \rho$. 
Also, by Lemma \ref{L:lambda<rho i<=lambda}, $\lambda' \leq \lambda$. 

If either $\lambda' < \lambda$ or $\rho' > \rho$ then $\epsilon(t')< \epsilon(t)=e$ so $t' \in F^{\epsilon(t')} \subseteq \Fred$ by the minimality of $e$ and \eqref{eqn:F^e chain of inclusion}, as needed.

So we assume that $\lambda'=\lambda$ and $\rho'=\rho$.
In particular 
\[
\epsilon(t')=\epsilon(t)=e.
\]
By Lemma \ref{L:lambda<rho i<=lambda} $\lambda'=\lambda$ only if $b<\lambda$ in which case $\dell(J_1(t)) <_L \dell(J_1(t'))$.

We now examine $\rho$ and $\rho'$.
If $\rho=\rho'=n$ then $J_0(t)=J_0(t')=\emptyset$ in which case $\delr(J_0(t))=\delr(J_0(t))$, both equal to the empty word.

So assume that $\rho=\rho'<n$.
Then $J_0(t), J_0(t') \neq \emptyset$ and in particular $\rho' \in K$.
So $\rho \in K$.
Hence, either $\rho \leq i-1$ or $\rho \geq b+1$.
If $\rho \geq b+1$ then $\rho>i-1$ since $i\leq b$, so $K=[\rho,n-1]$.
This implies that $J_0(t')=J_0(t)$ since $J_0(t) \subseteq [\rho,n-1]$, and therefore $\delr(J_0(t))=\delr(J_0(t'))$.
If $\rho \leq i-1$ then $\rho < b+1$ since $i\leq b$, and therefore $K=[\rho,i-1] \cup [b+1,n-1]$.
Lemma \ref{L:J cap K delr} implies that $\delr(J_0(t)) \leq_L \delr(J_0(t'))$.

To conclude, in Case 1 we get $\dell(J_1(t)) <_L \dell(J_1(t'))$ and  $\delr(J_0(t)) \leq_L \delr(J_0(t'))$ (in either case $\rho=\rho'=n$ or $\rho=\rho'<n$).
Hence, $\omega(t) \prec_L \omega(t')$.
Now, $t' \in F^e$ since we have already seen that  $\epsilon(t')=e$.
The maximality of $\omega(t)$ in $\Omega$ implies that $t' \in \Fred$, as needed.

\medskip
\noindent
{\em Case 2:} $\rho+1 \leq i \leq n$ and $i \neq k$.
Set $b$ and $t'$ as in Lemma \ref{L:lambda<rho i>rho}.
Thus, $J_0(t) \neq \emptyset$ (since $\rho<n$), and $\rho \leq b \leq i-1$, and $t'=\partial_b(\sigma_i(t))$ and $t' \in F(b)$.
Set $\lambda'=\lambda(t')$ and $\rho'=\rho(t')$.
By Lemma \ref{L:lambda<rho i>rho} $J_1(t')=J_1(t) \cap K$ where 
\[
K = [1,\min\{\lambda,b-1\}] \cup [i,\lambda].
\]
Dy Definition \ref{def:J_0 J_1 rho lambda}, $\lambda' \leq \lambda$.
Also, by  Lemma \ref{def:J_0 J_1 rho lambda}, $\rho' \geq \rho$. 

If either $\lambda' < \lambda$ or $\rho' > \rho$ then $\epsilon(t')< \epsilon(t)$ so $t' \in \Fred$ by the minimality of $e$ and \eqref{eqn:F^e chain of inclusion}, and we are done.
So we assume that $\lambda'=\lambda$ and $\rho'=\rho$.
In particular 
\[
\epsilon(t')=\epsilon(t)=e.
\]
By Lemma \ref{L:lambda<rho i>rho} $\rho'=\rho$ only if $b>\rho$ in which case $\delr(J_0(t)) <_L \delr(J_0(t'))$.

If $\lambda=\lambda'=0$ then $J_1(t)=J_1(t')=\emptyset$ in which case $\dell(J_1(t'))=\dell(J_1(t))$, both are equal to the empty word.
So assume that $\lambda=\lambda'>0$.
Then $J_1(t') \neq \emptyset$ and in particular $\lambda' \in K$.
So $\lambda \in K$.
Therefore either $\lambda \leq b-1$ or $i \leq \lambda$.

If $\lambda \leq b-1$ then $\lambda < i$ since $b \leq i-1$ and therefore $K=[1,\lambda]$.
By Definition \ref{def:J_0 J_1 rho lambda} $J_1(t) \subseteq [1,\lambda]$ so $J_1(t')=J_1(t)$ in this case, and then $\dell(J_1(t'))=\dell(J_1(t))$.
If $i \leq \lambda$ then $b-1\leq i-2<i\leq \lambda$ so $K=[1,b-1] \cup [i,\lambda]$.
Lemma \ref{L:J cap K dell} implies that $\dell(J_1(t)) \leq_L \dell(J_1(t'))$.

To conclude, in Case 2 we get $\delr(J_0(t)) <_L \delr(J_0(t'))$ and $\dell(J_1(t)) \leq_L \dell(J_1(t'))$ (in either case $\lambda=\lambda'=0$ or $\lambda=\lambda'>0$).
It follows that $\omega(t) \prec_L \omega(t')$.
We have seen that $\epsilon(t')=e$ so $t' \in F^e$.
The maximality of $\omega(t)$ in $\Omega$ implies that $t' \in \Fred$, and we are done.
\end{proof}

\bibliography{bibliography}{}

\begin{thebibliography}{10}

\bibitem{MR1035492}
Anders Bj\"{o}rner.
\newblock The {M}\"{o}bius function of subword order.
\newblock In {\em Invariant theory and tableaux ({M}inneapolis, {MN}, 1988)},
  volume~19 of {\em IMA Vol. Math. Appl.}, pages 118--124. Springer, New York,
  1990.

\bibitem{MR0690055}
Anders Bj\"{o}rner and Michelle Wachs.
\newblock On lexicographically shellable posets.
\newblock {\em Trans. Amer. Math. Soc.}, 277(1):323--341, 1983.

\bibitem{MR4071308}
Wojtek Chacholski, Ran Levi, and Roy Meshulam.
\newblock On the topology of complexes of injective words.
\newblock {\em J. Appl. Comput. Topol.}, 4(1):29--44, 2020.

\bibitem{MR0529895}
Frank~D. Farmer.
\newblock Cellular homology for posets.
\newblock {\em Math. Japon.}, 23(6):607--613, 1978/79.

\bibitem{MR1867354}
Allen Hatcher.
\newblock {\em Algebraic topology}.
\newblock Cambridge University Press, Cambridge, 2002.

\bibitem{MR2552261}
Jakob Jonsson and Volkmar Welker.
\newblock Complexes of injective words and their commutation classes.
\newblock {\em Pacific J. Math.}, 243(2):313--329, 2009.

\bibitem{MR2155519}
Moritz~C. Kerz.
\newblock The complex of words and {N}akaoka stability.
\newblock {\em Homology Homotopy Appl.}, 7(1):77--85, 2005.

\bibitem{MR0084138}
John Milnor.
\newblock The geometric realization of a semi-simplicial complex.
\newblock {\em Ann. of Math. (2)}, 65:357--362, 1957.

\bibitem{MR3032101}
Oscar Randal-Williams.
\newblock Homological stability for unordered configuration spaces.
\newblock {\em Q. J. Math.}, 64(1):303--326, 2013.

\bibitem{MR2043333}
Victor Reiner and Peter Webb.
\newblock The combinatorics of the bar resolution in group cohomology.
\newblock {\em J. Pure Appl. Algebra}, 190(1-3):291--327, 2004.

\bibitem{MR0300281}
C.~P. Rourke and B.~J. Sanderson.
\newblock {$\Delta$}-sets. {I}. {H}omotopy theory.
\newblock {\em Quart. J. Math. Oxford Ser. (2)}, 22:321--338, 1971.

\bibitem{MR0586429}
Wilberd van~der Kallen.
\newblock Homology stability for linear groups.
\newblock {\em Invent. Math.}, 60(3):269--295, 1980.

\end{thebibliography}
\bibliographystyle{plain}

\end{document}